\documentclass{siamltex}
\usepackage{amsmath,amssymb,amsfonts,latexsym,graphicx,color,multirow,footnote}
\usepackage{color}

\newtheorem{algorithm}{Algorithm}

\title{A randomized FEAST algorithm for  generalized eigenvalue problems }

\author{Guojian Yin\thanks{School of Mathematics, Sun Yat-sen University, Guangzhou, P. R. China ({\tt guojianyin@gmail.com}).}   }
\begin{document}
\maketitle

\begin{abstract} The FEAST algorithm, due to Polizzi, is a typical contour-integral based eigensolver  for computing the eigenvalues, along with their eigenvectors, inside a given region in the complex plane. It was formulated under the circumstance that the considered eigenproblem is Hermitian. The FEAST algorithm is stable and accurate, and has attracted much attention in recent years. However, it was observed that the FEAST algorithm may fail to find the target eigenpairs when applying it to the non-Hermitian problems. Efforts have been made to adapt the FEAST algorithm to non-Hermitian cases. In this work, we develop a new non-Hermitian scheme for the FEAST algorithm. The mathematical framework will be established, and the convergence analysis of our new method will  be studied. Numerical experiments are
reported to demonstrate the effectiveness of our method and to validate the convergence properties.
\end{abstract}



\section{Introduction}
Large-scale non-Hermitian eigenvalue problems arise in various areas of
science and engineering, such as  dynamic analysis of structures \cite{GLS94}, linear stability analysis of the Navier-Stokes equation in fluid dynamics \cite{CST00}, the electron energy and position problems in quantum chemistry \cite{FH74}, and resonant state calculation \cite{Reinhardt82}. In most practical  applications, it is not the whole spectrum but rather a significant part of it is of interest to the users \cite{NH13}. For example, 
in the model reduction of a linear dynamical system, one only needs to know
the response over a range of frequencies, see \cite{BDDRV00, GGD96}.

Consider the generalized eigenvalue problem
\begin{equation}\label{eq:1-1}
 A {\bf x} = \lambda  B {\bf x},
\end{equation}
where  $ A,  B \in {\mathbb C}^{n \times n}$.  The scalars $\lambda \in \mathbb{C}$
and the associated vectors ${\bf x} \in {\mathbb C}^n, {\bf x}\neq 0$, are
called the eigenvalues and their associated (right) eigenvectors, respectively \cite{BDDRV00, demmel, gvl}.  When $B$ is the identity matrix, then (\ref{eq:1-1}) becomes a standard eigenvalue problem. 
In this work, our objective is to compute the eigenvalues of (\ref{eq:1-1}) inside a given region in the complex plane, along with the corresponding eigenvectors.

Computing the partial spectrum of a large-scale problem is very difficult in practice. Maybe the most straightforward method is first using  the well-known QZ method \cite{ MS73} to compute the whole spectrum and then selecting the target eigenvalues. This direct method costs about $\mathcal{O}(n^3)$ \cite{gvl}, consequently, it is prohibitively expensive when the size of considered problem is large. In the past decades, the most successful methods for solving the partial spectrum of a large eigenproblem are based on the projection techniques \cite{BDDRV00, BDG97, stewart}, of which perhaps the Krylov subspace approaches are the most wildly used \cite{Ruhe98, saad}. However, the existing projection methods mainly focus on computing the extreme eigenvalues \cite{STL11} or the eigenvalues close to a given shift \cite{GLS94}. 

Recently, a class of eigensolvers based on contour integrals were proposed for computing the eigenvalues inside a given region in the complex plane \cite{AT14, Beyn12, DPS15, polizzi, SFT, ss, ST07}. Unlike the well-known Krylov subspace methods, these new methods use specifically defined contour integrals to generate subspaces to contain the eigenspace corresponding to the target eigenvalues. Then  the projection techniques are used to extract the target eigenpairs. Two typical examples of these contour-integral based eigensolvers are the Sakurai-Sugiura (SS) method \cite{ss} and the FEAST algorithm developed by Polizzi in \cite{polizzi}. By noticing that the SS method always suffers from numerical instability \cite{AT14, ISN10}, Sakurai {\it et al.}  turned to use the Rayleigh-Ritz procedure to extract the target eigenpairs, and leaded to a more stable contour-integral based eigensolver, called CIRR \cite{IS10, ST07}.

The derivation of both CIRR and FEAST is under the assumptions that $A$ and $B$ are Hermitian matrices and $B$ is positive definite, i.e., (\ref{eq:1-1}) is a Hermitian problem. 
It was shown in \cite{YCY14} that CIRR and FEAST may fail to find the target eigenpairs when (\ref{eq:1-1}) is a non-Hermitian problem. Motivated by this observation, the authors in \cite{YCY14} developed a non-Hermitian FEAST algorithm to make the FEAST algorithm also applicable for the non-Hermitian problems. Instead of the orthogonal projection technique used in the FEAST algorithm, the non-Hermitian FEAST algorithm proposed in \cite{YCY14} uses to the oblique projection technique with appropriately chosen left subspace to extract desired eigenpairs.

In this work, we would like to formulate another non-Hermitian scheme for the FEAST algorithm. We find that the FEAST algorithm can deal with the non-Hermitian problems if the left subspace spanned by a random matrix. The theoretical analysis will be given to justify our findings. The convergence properties also will be studied to show the effectiveness of our method.

The paper is organised as follows. In Section 2, 
we briefly describe  the FEAST algorithm \cite{polizzi}. In Section 3, we review the non-Hermitian variant of the FEAST algorithm proposed in \cite{YCY14}. We formulate our new non-Hermitian FEAST algorithm and give convergence analysis in Section 4.
In Section 5, numerical experiments are reported to illustrate the numerical performance of our method.

Throughout the paper, we use the following notation and terminology.
The subspace spanned by the columns of a matrix $ X$ is
denoted by ${\rm span}\{ X\}$. The rank and conjugate transpose of $X$ are
denoted by $\rank(X)$ and $X^*$ respectively.  The algorithms are presented in \textsc{Matlab} style.

\section{Introduction to FEAST}
In this section, we provide a brief review of the FEAST algorithm
{\cite{polizzi}}. The algorithm was formulated under the assumptions that $A$ and $B$ are Hermitian and
$B$ is positive definite, in which case the eigenvalues of (\ref{eq:1-1}) are real-valued \cite{demmel}.
The FEAST algorithm was developed for finding all eigenvalues of
(\ref{eq:1-1}) within a specified interval, say $[\sigma_1,\sigma_2]$,
and their associated eigenvectors. Without loss of generality, assume that the eigenvalues inside $[\sigma_1,\sigma_2]$ are  $\lambda_1\le \lambda_2\le\ldots \le\lambda_s$. Therefore, there are $s$ eigenvalues inside $[\sigma_1,\sigma_2]$.

Essentially, the FEAST algorithm
belongs to the family of subspace iteration with orthogonal projection \cite{TP13}.
Unlike the better known Krylov subspace methods, the FEAST
algorithm constructs a  subspace that envelops the desired
eigenspace via the contour integral defined as
\begin{equation}\label{eq:2-4}
 V:= \dfrac{1}{2\pi \sqrt{-1}}\oint_{\Gamma}  (z B- A)^{-1}  dz Y,
\end{equation}
where $\Gamma$ is any contour that contains $\{\lambda_i\}_{i=1}^s$ inside, and $Y$ is an $n\times s$ random matrix.   To formulate the FEAST algorithm, we need the following theorem.
\begin{theorem}[\cite{stewart}] \label{thm:2.2}
Let $ A$ and $ B$ be  $n\times n$ Hermitian matrices and that $ B$ is positive definite.
Then there exists an $n\times n$ matrix $ X=[{\bf x}_1, {\bf x}_2,\ldots, {\bf x}_n]$ for which
\begin{equation}\label{eq:2-1}
 X^* B X=  I_n \quad \mbox{and} \quad  X^* A X= \Lambda ={\rm diag}(\lambda_1,\lambda_2
\cdots, \lambda_n),
\end{equation}
where  $I_n$ is the $n \times n$ identity
matrix, $\{\lambda_i\}_{i=1}^n$
are the eigenvalues of the matrix pencil $zB-A$, and the
columns $\{{\bf x}_i\}_{i=1}^n$ of $ X$ are their associated eigenvectors.
\end{theorem}

By (\ref{eq:2-1}) and the residue theorem in complex analysis \cite{Rudin}, we have
\begin{equation}\label{eq:2-5}
  V=  \dfrac{1}{2\pi \sqrt{-1}} X\oint_{\Gamma}(z I_n-\Lambda)^{-1}dzY
 (BX)^{-1}=  X_{(:,1:s)}( X_{(:,1:s)})^{*}Y.
\end{equation}
 Then the columns of $ V$ form a basis for the eigenspace ${\rm span}\{X_{(:,1:s)}\}$, if $( X_{(:,1:s)})^{*}Y$ is  full-rank. Forming the $s\times s$ matrices $\hat{A}= V^* A  V$ and $\hat{B}=  V^* B  V$, solving
the problem (\ref{eq:1-1}) now is reduced to computing the eigenpairs of the projected eigenvalue problem 
\begin{equation}\label{eq:11-3-1}
\hat{A} {\bf y} = \lambda \hat{B} {\bf y},
\end{equation}
according to the Rayleigh-Ritz procedure \cite{polizzi, stewart}.

To generate the projected eigenproblem (\ref{eq:11-3-1}), the most important task is to compute the basis vectors $V$.  In view of  (\ref{eq:2-4}) and (\ref{eq:2-5}), we know that $V$ has to be computed numerically by a quadrature scheme. Let $\Gamma$ be the circle with center at $\gamma = (\sigma_1+\sigma_2)/2$
and radius $\rho = (\sigma_2-\sigma_1)/2$, applying the $q$-point Gauss-Legendre quadrature \cite{DR84} to compute $V$ numerically yields
\begin{equation}\label{eq:8-4}
  V = \dfrac{1}{2\pi \sqrt{-1}}\oint_{\Gamma}  (z B- A)^{-1}dz Y
 \approx \frac{1}{2} \sum^{q}_{j=1}\omega_j(z_j-\gamma)(z_j  B- A)^{-1}  Y,
\end{equation}
where $z_j=\gamma+\rho e^{\sqrt{-1}\theta_j}$, $\theta_j=(1+ t_j)\pi$, and
$t_j$ is the $j$th Gaussian node with associated weight $\omega_j$. From (\ref{eq:8-4}), one can see that the dominant computational work of the FEAST algorithm is solving the linear systems of the form
\begin{equation}\label{eq:11-4-1}
(z_jB-A)X_j = Y,\quad  j = 1,\ldots, q.
\end{equation}

The complete FEAST algorithm is given as follows.

\begin{algorithm}\label{alg:8}
Input Hermitian matrices $A$ and $B$ with $B$ being positive definite,
a uniformly-distributed random matrix $Y \in \mathbb{R}^{n\times t}$,
where $t\geq s$,
the circle $\Gamma$ enclosing the interval $[\sigma_1,\sigma_2]$, and
a convergence  tolerance $\epsilon$.
The function ``{\sc Feast}" computes eigenpairs $(\hat{\lambda}_i, \hat{\bf x}_i)$ of
(\ref{eq:1-1}) that satisfy
\begin{equation}\label{con-cre}
\hat{\lambda}_i \in [\sigma_1,\sigma_2]
\quad {\rm and} \quad \sum_{i = 1}^s \hat{\lambda}_i< \epsilon,
\end{equation}
and they are output in the vector $\Lambda_s$ and the matrix $X_s$.
\end{algorithm}
\vspace{.2cm}
%\vspace{-5pt}
\begin{tabbing}
x\=xxx\= xxx\=xxx\=xxx\=xxx\=xxx\kill
\> Function $[ \Lambda_s,  X_s]$ = {\sc Feast}$( A,  B,  Y, \Gamma, \epsilon)$\\
\>1.\> Compute $ V $ approximately by (\ref{eq:8-4}).\\
\>2.\> Set $\hat{A}=  V^* A  V$ and $\hat{B}=  V^* B  V$.\\
\>3.\> Solve the generalized eigenproblem of size $t$: $\hat{A} {{\bf y}}=\hat{\lambda}  \hat{B} {{\bf y}}$,
to obtain the \\
\>\>  eigenpairs $\{(\hat{\lambda}_i, {\bf y}_i)\}_{i = 1}^{t}$.\\
\>4.\> Compute $\hat{\bf {x}}_i=   V{{\bf y}}_i, i =1,2,\ldots t$.\\
\>5.\> Check if $\{(\hat{\lambda}_i, \hat{\bf x}_i)\}_{i=1}^t$ satisfy the convergence criteria (\ref{con-cre}). If $s$ eigenpairs \\
\>\> satisfy (\ref{con-cre}), stop. Otherwise, set $X_t=  [\hat{\bf x}_1,\hat{\bf x}_2,\ldots,\hat{\bf x}_{t}]$ and $Y=BX_t$,  then\\
\>\> go back to Step 1.
\end{tabbing}
\vspace{.2cm}

The FEAST algorithm is an accurate and reliable technique \cite{kramer, TP13}. It transforms the difficulty of  solving the eigenproblem (\ref{eq:1-1}) to that of  solving linear systems (\ref{eq:11-4-1}). Since the quadrature nodes $z_j$ and the columns of the right-hand sides in (\ref{eq:11-4-1}) are independent, the FEAST algorithm can be easily implemented on parallel machines. Due to these appealing features, the FEAST algorithm attracts much attention recently.

\section{A non-Hermitian FEAST algorithm}
The FEAST algorithm was formulated when (\ref{eq:1-1}) is a Hermitian problem. However, when it comes to the non-Hermitian problem, it was found  in \cite{YCY14}  that the FEAST algorithm may fail to compute the desired eigenpairs; a simple example was given  to illustrate this fact. Motivated by this observation,  the authors in \cite{YCY14} developed a non-Hermitian FEAST algorithm so as to adapt FEAST to the non-Hermitian cases. The key to the success of their non-Hermitian FEAST  algorithm is that  the oblique projection technique, instead of the orthogonal projection technique used in FEAST, with appropriately chosen left subspace is used to extract the desired eigenpairs.

The only requirement for the non-Hermitian FEAST algorithm proposed in \cite{YCY14} is that the matrix pencil $z B-A$ is regular, which means the method is able to deal with the most common generalized eigenproblems \cite{BDDRV00}.  Recall that a matrix pencil $z B-A$ is regular if ${\rm det}(zB-A)$ is not identically zero for all $z \in \mathbb{C}$. As with the Jordan canonical form for a matrix, there exists a canonical form for the regular matrix pencil $zB-A$.

\begin{theorem} [The Weierstrass canonical form \cite{G59, YCY14}] \label{thm2.1}
Let $zB-A$ be a regular matrix pencil of order $n$. Then there exist nonsingular matrices $S$ and $T \in \mathbb{C}^{n\times n}$ such that
\begin{equation}\label{eq:8-3-1}
TAS = \begin{bmatrix}
  J_d    & 0   \\
   0   & I_{n-d}
\end{bmatrix}  \quad {\rm and} \quad TBS= \begin{bmatrix}
   I_d   & 0   \\
  0    & N_{n-d}
\end{bmatrix},
\end{equation}
where $J_d$ is a $d\times d$ matrix in Jordan canonical form
with its diagonal entries corresponding to the eigenvalues of $zB-A$, $N_{n-d}$ is an $(n-d)\times (n-d)$ nilpotent matrix also in Jordan canonical form, and $I_d$ denotes the identity matrix of order $d$.
\end{theorem}

Let $J_d$ be of the form
\begin{equation}\label{equ:7-17-1}
J_d =   \left[ \begin{array}{cccc}
 J_{d_1}(\lambda_1) & 0 & \cdots & 0\\
0 &  J_{d_2}(\lambda_2) & \cdots & 0\\
\vdots & \vdots& \ddots &  \vdots\\
0 & 0 & \cdots &  J_{d_m}(\lambda_m)
\end{array}
\right]
\end{equation}
where $\sum_{i=1}^m d_i  =d$ and
$ J_{d_i}(\lambda_i)$ are $d_i \times d_i$ matrices of the form
$$\begin{array}{ccc}
 J_{d_i}(\lambda_i) = \left[ \begin{array}{ccccc}
\lambda_i & 1 & 0& \cdots  & 0\\
0 & \lambda_i & 1 & & \vdots\\
 & \ddots & \ddots & \ddots &0\\
\vdots & & \ddots &\ddots &  1\\
0 &  \cdots && 0&\lambda_i
\end{array}
\right], & & i = 1, 2, \ldots, m
\end{array}
$$
with $\lambda_i$ being the eigenvalues. Here the $\lambda_i$
are not necessarily distinct and can be repeated according to
their multiplicities.

Let $N_{n-d}$ be of the form
\begin{equation*}\label{equ:7-17-11}
N_{n-d} =   \left[ \begin{array}{cccc}
N_{d^{\prime}_1} & 0 & \cdots & 0\\
0 &  N_{d^{\prime}_2} & \cdots & 0\\
\vdots & \vdots& \ddots &  \vdots\\
0 & 0 & \cdots &  N_{d^{\prime}_{m^{\prime}}}
\end{array}
\right],
\end{equation*}
where $\sum_{i=1}^{m^{\prime}} d^{\prime}_i  =n-d$ and
$N_{d^{\prime}_i}$ are $d^{\prime}_i \times d^{\prime}_i$ matrices of the form
$$\begin{array}{ccc}
N_{d^{\prime}_i} = \left[ \begin{array}{ccccc}
0 & 1 & 0& \cdots  & 0\\
0 &0 & 1 & & \vdots\\
 & \ddots & \ddots & \ddots &0\\
\vdots & & \ddots &\ddots &  1\\
0 &  \cdots && 0&0
\end{array}
\right], & & i = 1, 2, \ldots, m^{\prime}.
\end{array}
$$

Partition $S$ into block form
$S = [S_1, S_2, \ldots, S_ m, S_{m+1}]$,
where each $S_i\in \mathbb{C}^{n\times d_i}$, $1\le i\le  m$,  and $S_i$ into $S_i=[{\bf s}_1^i, {\bf s}_2^i, \ldots ,{\bf s}_{d_i}^i]$ with
${\bf s}_j^i \in {\mathbb C}^{n}$, $1\leqslant j \leqslant d_i$.
It was verified in \cite{YCY14} that for any eigenvalue $\lambda_i$,
$1\le i \le m$,
\begin{equation}\label{eq:11-8-2}
(\lambda_i B-A)S \begin{bmatrix}
  I_d &   0 \\
    0  &  N_{n-d}
\end{bmatrix} =  BS\begin{bmatrix}
  \lambda_i I_d-J_d   &   0 \\
    0  &  \lambda_iN_{n-d}-I_{n-d}
\end{bmatrix}.
\end{equation}
By comparing the first $d$ columns on both sides above, we get
\begin{equation}\label{Jordan}
(\lambda_i B-A){\bf s}_j^{i} = B{\bf s}_{j-1}^{i}, \quad 1\le j \le d_i,
\quad 1 \le i \le m,
\end{equation}
with ${\bf s}_0^{i} \equiv {\bf 0}$.
We can see that ${\bf s}_1^{i}$ are the eigenvectors corresponding
to the eigenvalues $\lambda_i$ for all $1 \le i \le m$.

Let $\Gamma$ be a  positively oriented simple closed curve enclosing the desired eigenvalues. Again without loss of generality, we let the eigenvalues of (\ref{eq:1-1}) enclosed by $\Gamma$ be $\{\lambda_1, \ldots, \lambda_l\}$, and $s: = d_1+d_2+\cdots + d_{l}$ be the number of eigenvalues inside $\Gamma$ with multiplicity taken into account. 
Define the contour integral
\begin{equation} \label{eq:2-2-21}
  Q:= \dfrac{1}{2\pi \sqrt{-1}}\oint_{\Gamma} (z B- A)^{-1} Bdz.
\end{equation}
According to the residue theorem in complex analysis \cite{Rudin}, it was verified in \cite{YCY14} that
\begin{equation}\label{equ:7-26-4}
 Q = \frac{1}{2\pi \sqrt{-1}} \oint_\Gamma  (zB-A)^{-1}B  dz =  S\left[ \begin{array}{cc}
  I_s & 0\\
0 & 0
 \end{array} \right]  S^{-1} =  S_{(:,1:s)} ( S^{-1})_{(1:s,:)}.
\end{equation}
One can show that $Q^2 = Q$, which means $Q$ is a projector onto subspace $\mathcal{K} = {\rm span}\{S_{(:,1:s)}\}$. Define 
\begin{equation}\label{eq:11-5-2}
U:= QY = S_{(:,1:s)} ( S^{-1})_{(1:s,:)}Y,
\end{equation}
where $Y$ is an $n\times s$ random matrix. Therefore, $U$ is the projection of $Y$ onto the subspace $\mathcal{K}$.  Now we would like to show that the columns of $U$ form a basis for the subspace $\mathcal{K}$. We begin with
\begin{lemma}[\cite{YCY14}]\label{lem:8-8-1}
Let $Y\in\mathbb{R}^{n\times s}$. If the entries of
$Y$ are random numbers from a continuous distribution and that they are independent and identically distributed (i.i.d.), then with probability 1, the matrix
$( S^{-1})_{(1:s,:)}Y$ is nonsingular.
\end{lemma}

According to (\ref{eq:11-5-2}) and {Lemma} \ref{lem:8-8-1}, we can conclude that the columns of $U$  form a basis for the subspace $\mathcal{K}$. Note that $\mathcal{K}$ contains the eigenspace corresponding to the desired eigenvalues (see (\ref{Jordan}) for details), it is natural to take $\mathcal{K}$ as the right subspace. The FEAST algorithm takes advantage of the often used orthogonal projection technique to extract desired eigenpairs.  The authors in \cite{YCY14} found that this extraction approach may fail to compute the desired eigenpairs  for the FEAST algorithm when (\ref{eq:1-1}) is a non-Hermitian problem. To address this deficiency, they resorted to the oblique projection method, and developed a non-Hermitian FEAST algorithm. In their method, the left subspace is taken as $B\mathcal{K}$;  the approximate eigenpairs $(\bar{\lambda}, \bar{{\bf x}})$ are obtained by imposing the Petrov-Galerkin condition \cite{BDDRV00, saad}:
\begin{equation}\label{eq:11-5-3}
(A\bar{{\bf x}}-\bar{\lambda}B\bar{{\bf x}})\  \bot \ B\mathcal{K},
\end{equation}
where $\bar{\lambda}\in \mathbb{C}$ and $\bar{{\bf x}}\in \mathcal{K}$. It was shown in \cite{YCY14} that the columns of $BU$ form a basis for $B\mathcal{K}$. Therefore (\ref{eq:11-5-3}) can be written in matrix form
\begin{equation}\label{eq:6-3-12}
(BU)^*(AU{\bf y}-\bar{{\lambda}}BU{\bf y})=0,
\end{equation}
where ${\bf y} \in \mathbb{C}^s$ satisfying $\bar{{\bf x}}=U{\bf y}$.
Accordingly, solving the eigenvalues of (\ref{eq:1-1}) inside $\Gamma$ now is reduced to solve the projected eigenproblem
\begin{equation}\label{eq:6-3-13}
\overline{ A}{\bf y}=\bar{\lambda}\overline{ B}{\bf y},
\end{equation}
with 
\begin{equation}\label{equ:7-10}
\overline{ A}=  (BU)^* A U\quad \mbox{and}\quad \overline{ B}=  (BU)^* B U.
\end{equation}
The key to the success of the non-Hermitian FEAST algorithm proposed in \cite{YCY14} is that the left subspace is taken as $B\mathcal{K}$, instead of $\mathcal{K}$ used in the FEAST algorithm. Due to this, below we call this non-Hermitian FEAST algorithm BFEAST for the ease of reference. The following theorem justifies their choice of the left subspace.

%\vspace{2mm}
\begin{theorem}\label{Th:8-3-3}
Let $\{(\bar{\lambda}_i, {\bf y}_i)\}_{i=1}^s$ be the eigenpairs of
the projected eigenproblem (\ref{eq:6-3-13}). Then
$\{(\bar{\lambda}_i, U{\bf y}_i)\}_{i=1}^s$ are the eigenpairs of (\ref{eq:1-1})  located inside ${\Gamma}$.
\end{theorem}

In order to generate the projected eigenproblem (\ref{eq:6-3-13}), the most important task is to compute the projection $U$  (see (\ref{eq:11-5-2})). In practice, we have to compute $U$  approximately by a quadrature rule:
\begin{equation}\label{eq:2-13}
 U =QY=\frac{1}{2\pi \sqrt{-1}} \oint_\Gamma (zB-A)^{-1}BdzY \approx\frac{1}{2\pi \sqrt{-1}} \sum^{q}_{j=1}\omega_j(z_j  B- A)^{-1}B  Y,
\end{equation}
where $z_j$ are the quadrature nodes on $\Gamma$ associated with weights $\omega_j$. From (\ref{eq:2-13}), we know that the dominant work is  solving $q$ linear systems of the form
\begin{equation}\label{eq:11-9-1}
(z_iB-A)X_i = BY.
\end{equation}

The non-Hermitian FEAST algorithm (BFEAST) can be described as follows.
\begin{algorithm}\label{alg:5}
Input $ A,  B \in {\mathbb C}^{n \times n}$, an i.i.d. random
matrix $Y \in {\mathbb R}^{n \times  t}$
where $ t\geq s$, a closed curve $\Gamma$, a convergence tolerance
$\epsilon$, and ``${\texttt{{max\_iter}}}$" to control the
maximum number of iterations. The function ``{\sc BFEAST}"
computes eigenpairs $(\bar{\lambda}_i, \bar{{\bf x}}_i)$ of (\ref{eq:1-1}) that satisfies
\begin{equation}\label{con-cre_2}
\bar{\lambda}_i \ {\rm inside} \ \Gamma
\quad {\rm and} \quad \frac{\| A\bar{{\bf x}}_i - \bar{\lambda}_i  B\bar{{\bf x}}_i\|_2}{\| A \bar{{\bf x}}_i\|_2+
\| B \bar{{\bf x}}_i\|_2} < \epsilon.
\end{equation}
The results are stored in the vector $\Lambda_s$ and the matrix $X_s$. \end{algorithm}
\vspace{.2cm}
%\vspace{-5pt}
\begin{tabbing}
x\=xxx\= xxx\=xxx\=xxx\=xxx\=xxx\kill
\> Function $[ \Lambda_s,  X_s] = \textsc{BFEAST}( A,  B,  Y, \Gamma, \epsilon, \texttt{max\_iter})$\\
\>1.\>For $k = 1,\cdots, \texttt{max\_iter}$\\
\>2.\>\> Compute $U$  approximately by the quadrature rule (\ref{eq:2-13}). \\
\>3.\>\> Compute QR decompositions: $U = U_1R_1$ and $BU = U_2R_2.$\\
\>4.\>\>Form $\overline{A} =  U_2^*  A  U_1$ and $\overline{B} =  U_2^*  B  U_1$.\\
\>5.\>\> Solve the projected eigenproblem $\overline{A} {\bf y} = \bar{\lambda} \overline{B} {\bf y}$ of size $ t$ to obtain eigenpairs\\
\>\>\>$\{(\bar{\lambda}_i, {\bf y}_i)\}_{i=1}^{t}$. Set  $ \bar{{\bf x}}_i =  U_1{\bf y}_i, i = 1, 2, \ldots,  t$. \\
\>6.\>\> Set $ \Lambda_s=\left[ \  \right]$ and $X_s=\left[ \  \right]$.\\
\>7.\>\> For $i = 1: t$\\
\>8.\>\>\>If $(\bar{\lambda}_i,\bar{ {\bf x}}_i)$ satisfies (\ref{con-cre_2}), then $ \Lambda_s = [\Lambda_s, \bar{\lambda}_i]$ and $ X_s =[X_s, \bar{{\bf x}}_i]$.\\
\>9.\>\>End\\
\>10.\>\>If there are $s$ eigenpairs satisfying (\ref{con-cre_2}), stop. Otherwise, set $Y = U_1$.\\
\>11.\>End. \\
\end{tabbing}

\section{A randomized FEAST algorithms} In the previous two sections, we reviewed the FEAST algorithm, as well as its non-Hermitian variation, i.e., the BFEAST algorithm.  In this part, we first formulate another non-Hermitian scheme for the FEAST algorithm.  After that, the convergence analysis  will be given to illustrate the effectiveness of our new method. 

\subsection{The derivation of our method} In \cite{YCY14}, the authors used the oblique projection technique, rather than the wildly used orthogonal projection technique, to extend FEAST to the non-Hermitian problems. The key step is that they take the left subspace to $B\mathcal{K}$ instead of $\mathcal{K}$ used in the original FEAST algorithm. Here we present another scheme for the non-Hermitian FEAST algorithm. The intuition behind our new method is inspired by Lemma \ref{lem:8-8-1}. The following theorem validates our intuition.

\vspace{2mm}
\begin{theorem}\label{Th:10-2-1} Let
$R$ be an $n\times s$  random matrix, whose entries are independent and identically distributed (i.i.d.). Define
\begin{equation}\label{eq:10-2-2}
\widetilde{ A}=  R^* A U\quad \mbox{and}\quad \widetilde{ B}=  R^* B U.
\end{equation}
Let $\{(\tilde{\lambda}_i, {\bf y}_i)\}_{i=1}^s$ be the eigenpairs of
the projected eigenproblem
\begin{equation}\label{eq:10-2-1}
\widetilde{ A}{\bf y}=\tilde{\lambda}\widetilde{ B}{\bf y}.
\end{equation}
Then
$\{(\tilde{\lambda}_i, U{\bf y}_i)\}_{i=1}^s$ are the eigenpairs of (\ref{eq:1-1})  located inside ${\Gamma}$.
\end{theorem}
\begin{proof}
By (\ref{eq:8-3-1}), one can verify that 
\begin{equation}\label{eq:10-3-2}
(\tilde{\lambda}  B -  A) S_{(:,1:s)} =  (T^{-1})_{(:,1:s)} ( \tilde{\lambda}  I_{(1:s,1:s)} - J_{(1:s,1:s)}).
\end{equation}
By (\ref{eq:11-5-2}), (\ref{eq:10-2-2}) and (\ref{eq:10-3-2}), we have
\begin{eqnarray}\label{eq:10-31-1}
\tilde{\lambda} \widetilde{ B} - \widetilde{ A} & = & R^*(\tilde{\lambda}  B -  A)U \\ \nonumber
 & = &  R^*(\tilde{\lambda}  B -  A) S_{(:,1:s)} ( S^{-1})_{(1:s,:)}  Y \\ \nonumber
 & = & R^*(T^{-1})_{(:,1:s)} ( \tilde{\lambda}  I_{(1:s,1:s)} - J_{(1:s,1:s)})( S^{-1})_{(1:s,:)}  Y.
\end{eqnarray}
Therefore, the characteristic polynomial of the eigenproblem  (\ref{eq:10-2-1}) is
$$
\det (\tilde{\lambda} \widetilde{ B} - \widetilde{ A}) = \det( R^*
 (T^{-1})_{(:,1:s)}) \det(\tilde{\lambda} I_{(1:s,1:s)} - J_{(1:s,1:s)})\det (( S^{-1})_{(1:s,:)}  Y).
$$
By Lemma \ref{lem:8-8-1}, we know that $R^*
 (T^{-1})_{(:,1:s)}$ and $( S^{-1})_{(1:s,:)}  Y$ are nonsingular. As a result, due to the special structure of $ J_{(1:s,1:s)}$, the roots of the characteristic polynomial $\det (\tilde{\lambda} \widetilde{ B} - \widetilde{ A})$ are $\lambda_1, \lambda_2, \ldots, \lambda_l$ with multiplicities $d_1, d_2, \ldots, d_l$ respectively.

Recall that $\lambda_1,\ldots,\lambda_l$ are not necessary distinct. Without loss of generality, let us consider the case where  $\tilde{\lambda} = \lambda_1=\lambda_2$ and $\tilde{\lambda}\neq \lambda_i$ for $3\leqslant i \leqslant l$.   Since
$(\tilde{\lambda}\widetilde{ B} - \widetilde{ A}){\bf y} = 0$ and $R^*
 (T^{-1})_{(:,1:s)}$ is nonsingular, by (\ref{eq:10-31-1}) we have
$$
(\tilde{\lambda}  I_{(1:s,1:s)} - J_{(1:s,1:s)}) ( S^{-1})_{(1:s,:)}  Y {\bf y} =  0.
$$
 The above equation
in turn implies that
$( S^{-1})_{(1:s,:)}  Y{\bf y} = \alpha{\bf e}_{1}+\beta{\bf e}_{d_1+1}$
for some scalars $\alpha$ and $\beta$ not both zero due to the special structure of $ J$, where ${\bf e}_1$ and ${\bf e}_{d_1+1}$ are the first and the $(d_1+1)$th columns of the $s\times s$ identity matrix, respectively.
Therefore,
$$U {\bf y}
= S_{(:,1:s)} ( S^{-1})_{(1:s,:)}  Y{\bf y}
=  \alpha S_{(:,1:s)} {\bf e}_1+\beta S_{(:,1:s)} {\bf e}_{d_1+1}.$$
% In other words, $\alpha^{-1}  U \tilde{{\bf x}}$ is the first column of $ S$.
Note that the first and the $(d_1+1)$th columns of $ S$ are
the eigenvector of (\ref{eq:1-1})
corresponding to
the eigenvalues $\lambda_1$ and $\lambda_2$ respectively by (\ref{Jordan}). Thus, their linear combinations are the eigenvectors associated with $\tilde{\lambda}$. The proof is completed.

\end{proof}

Theorem \ref{Th:10-2-1} tells us that the FEAST algorithm can deal with the non-Hermitian eigenproblems if we take the left subspace spanned by a random matrix. Due to the usage of  random matrix, we call our new non-Hermitian FEAST algorithm RFEAST for ease of reference.

\begin{algorithm}\label{alg:6}
Input $ A,  B \in {\mathbb C}^{n \times n}$, an i.i.d. random
matrix $Y \in {\mathbb R}^{n \times  t}$
where $ t\geq s$, a closed curve $\Gamma$, a convergence tolerance
$\epsilon$, and ``${\texttt{{max\_iter}}}$" to control the
maximum number of iterations. The function ``{\sc RFEAST}"
computes eigenpairs $(\tilde{\lambda}_i, \tilde{{\bf x}}_i)$ of (\ref{eq:1-1}) that satisfies
\begin{equation}\label{con-cre_3}
\tilde{\lambda}_i \ {\rm inside} \ \Gamma
\quad {\rm and} \quad \frac{\| A\tilde{{\bf x}}_i - \tilde{\lambda}_i  B\tilde{{\bf x}}_i\|_2}{\| A \tilde{{\bf x}}_i\|_2+
\| B \tilde{{\bf x}}_i\|_2} < \epsilon.
\end{equation}
The results are stored in the vector $\Lambda_s$ and the matrix $X_s$. \end{algorithm}
\vspace{.2cm}
%\vspace{-5pt}
\begin{tabbing}
x\=xxx\= xxx\=xxx\=xxx\=xxx\=xxx\kill
\> Function $[ \Lambda_s,  X_s] = \textsc{RFEAST}( A,  B,  Y, \Gamma, \epsilon, \texttt{max\_iter})$\\
\>1.\>For $k = 1,\cdots, \texttt{max\_iter}$\\
\>2.\>\> Compute $U$  approximately by the quadrature rule (\ref{eq:2-13}). \\
\>3.\>\> Generate an $n\times t$ random matrix $R$, and compute QR decompositions: \\
\>\>\> $U = U_1R_1$ and $R = U_2R_2.$\\
\>4.\>\>Form $ \tilde{A} =  U_2^*  A  U_1$ and $\tilde{B} =  U_2^*  B  U_1$.\\
\>5.\>\> Solve the projected eigenproblem $\tilde{A} {\bf y} = \tilde{\lambda} \tilde{B} {\bf y}$ of size $ t$ to obtain eigenpairs\\
\>\>\>$\{(\tilde{\lambda}_i, {\bf y}_i)\}_{i=1}^{t}$. Set  $ \tilde{{\bf x}}_i =  U_1{\bf y}_i, i = 1, 2, \ldots,  t$. \\
\>6.\>\> Set $ \Lambda_s=\left[ \  \right]$ and $X_s=\left[ \  \right]$.\\
\>7.\>\> For $i = 1: t$\\
\>8.\>\>\>If $(\tilde{\lambda}_i,\tilde{ {\bf x}}_i)$ satisfies (\ref{con-cre_3}), then $ \Lambda_s = [\Lambda_s, \tilde{\lambda}_i]$ and $ X_s =[X_s, \tilde{{\bf x}}_i]$.\\
\>9.\>\>End\\
\>10.\>\>If there are $s$ eigenpairs satisfying (\ref{con-cre_3}), stop. Otherwise, set $Y = U_1$.\\
\>11.\>End. \\
\end{tabbing}
%\vspace{.2cm}

Our method can make the FEAST algorithm applicable to the non-Hermitian problems. Obviously, in each iteration, the dominant work in our method is to compute the projection $U$ by a quadrature scheme, see (\ref{eq:2-13}) for details.  As with other contour-integral based methods, our algorithm replaces the difficulty of solving the eigenvalue problem (\ref{eq:1-1}) by the difficulty of solving the linear systems (\ref{eq:11-9-1}), and has a good potential to be parallelized.

\subsection{Convergence Analysis}
In this part, we study the convergence properties of our new method (Algorithm \ref{alg:6}) to show its effectiveness.

An important quantity for the convergence properties of projection methods is the distance of the exact eigenvector from the search subspace \cite{saad}. We begin the convergence analysis of our method  from this perspective. For notational convenience, we represent the approximate projection computed in the $k$th iteration in Algorithm \ref{alg:6} by  $U^{(k)}$.

Compute the contour integral
\begin{equation}\label{eq:11-20-1}
f(\mu) =  \frac{1}{2\pi \sqrt{-1}}\oint_{\Gamma}\frac{1}{z-\mu}dz
\end{equation}
by a $q$-point quadrature rule on $\Gamma$:
\begin{equation}\label{eq:11-10-1}
f(\mu) \approx\tilde{f}(\mu) =\frac{1}{2\pi \sqrt{-1}} \sum^{q}_{i=1}\frac{\omega_i}{z_i-\mu},
\end{equation}
where $z_j$ are the quadrature nodes on $\Gamma$ associated with weights $\omega_j$.  Let $\Gamma$ be an unit circle with center at $\gamma$ and $\mu = \gamma +re^{\sqrt{-1}\theta}, \theta\in (-\pi, \pi]$. Therefore,  $\mu$ is located inside $\Gamma$ when $0 \leq r< 1$ and is located outside $\Gamma$ when $r> 1$. In theory, according to the  residue theorem, we have that $f(\mu)=1$ if $\mu$ is located inside $\Gamma$ and  $f(\mu)=0$ if $\mu$ is located outside $\Gamma$ \cite{Rudin}. Compute the approximation $\tilde{f}(\mu$) of  $f(\mu)$ by the Gauss-Legendre quadrature with 16 integration points on $\Gamma$.  \textsc{Fig} \ref{Fig:4-4-1} depicts the magnitude of $|\tilde{f}(\mu)|$. We can see that $|\tilde{f}(\mu)|$ is close 1 when $\mu$ is contained inside $\Gamma$ and is close to $0$ when $\mu$ is outside $\Gamma$. Without loss of generality, assume that 
\begin{equation}\label{eq:12-3-1}
|\tilde{f}(\lambda_1)| \geq |\tilde{f}(\lambda_2)|\geq \cdots \geq |\tilde{f}(\lambda_l)| >|\tilde{f}(\lambda_{l+1})|\geq\cdots \geq |\tilde{f}(\lambda_m)|. 
\end{equation}

\begin{figure}
\begin{center}
\includegraphics[width=13cm]{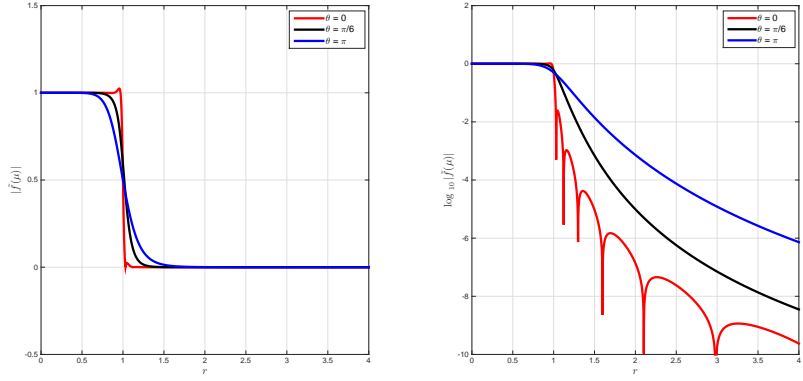}
\caption{Here $\Gamma$ is a unit circle with center at $\gamma$. The approximation $\tilde{f}(\mu)$, where $\mu = \gamma+r e^{\sqrt{-1}\theta}, \theta\in (-\pi, \pi]$, is computed by the Gauss-Legendre quadrature with $16$ quadrature nodes.  Therefore, $\mu$ is located inside $\Gamma$ when $0 \leq r< 1$ and is located outside $\Gamma$ when $r> 1$. The left picture shows the general shape of $|\tilde{f}(\mu)|$, while the right one shows the logarithmic scale shape of the function.}
\label{Fig:4-4-1}
\end{center}
\end{figure}

\begin{theorem}\label{Th:11-10-1}
Let $d_0 = 0$, then $S_{(:, 1+d_0+\ldots+d_{j-1})}$ is an eigenvector corresponding to $\lambda_j$. Let  $t$ be the size of starting vectors $Y$ satisfying $t> s$, then there exists an integer $l^{\prime}, l^{\prime}>l$, such that $\sum^{l^{\prime}-1}_id_i<t\leq \sum^{l^{\prime}}_id_i$. Suppose the eigenvalues outside $\Gamma$ are simple, which implies that $d_i = 1$ for $i = l+1, \ldots, m$.
There exists a vector $v_j^{(k)} \in {\rm span}\{U^{(k)}\}$, $j = 1, \ldots, l$, such that
\begin{equation}\label{eq:11-10-3}
\|S_{(:, 1+d_0+\ldots+d_{j-1})}-v_j^{(k)}\|_2 \leq \tau_j  \big(\dfrac{ |\tilde{f}(\lambda_{l^{\prime}})|}{|\tilde{f}(\lambda_{j})|}\big)^k,
\end{equation}
where $ \tau_j$ is a constant. In particular,
\begin{equation}\label{eq:11-10-3}
\|(I_n - Q_{(k)})S_{(:, 1+d_0+\ldots+d_{j-1})}\|_2  \leq \tau_j \big(\dfrac{ |\tilde{f}(\lambda_{l^{\prime}})|}{|\tilde{f}(\lambda_{j})|}\big)^k,
\end{equation}
where $Q_{(k)}$ is the orthogonal projector onto the subspace ${\rm span}\{U^{(k)}\}$.
\end{theorem}
\begin{proof} Let $U^{(0)} = Y$ be an $n\times t$ random matrix and $Z= (S^{-1})_{(1:t, :)}U^{(0)}$. By Lemma \ref{lem:8-8-1}, we know $Z$ is nonsingular, then
\begin{eqnarray}
U^{(0)} & = & S S^{-1}U^{(0)} = \left[S_{(:, 1:t)}(S^{-1})_{(1:t, :)}+S_{(:, t+1:n)}(S^{-1})_{(t+1:n, :)}\right]U^{(0)} \nonumber \\
 & = &  \left[S_{(:, 1:t)}+S_{(:, t+1:n)}(S^{-1})_{(t+1:n, :)} U^{(0)}Z^{-1}\right] Z \nonumber \\
  & = &\left[(S_{(:, 1:s)}+S_{(:, t+1:n)} E_{(0)}), V_{(0)}\right]Z, \label{eq:11-10-4}
\end{eqnarray}
where $E_{(0)}$ is the first $s$ columns of $(S^{-1})_{(t+1:n, :)} U^{(0)}Z^{-1}$, and $V_{(0)}$
is the last $(t-s)$ columns of matrix $S_{(:, 1:t)}+S_{(:, t+1:n)}(S^{-1})_{(t+1:n, :)}U^{(0)} Z^{-1}$.

Under the assumption that the eigenvalues outside $\Gamma$ are simple, the matrix $N_{n-d}$ in (\ref{eq:8-3-1}) is a zero matrix. Let
\begin{equation}\label{eq:11-18-1}
D =  \frac{1}{2\pi \sqrt{-1}} \sum^{q}_{j=1}\omega_j\begin{bmatrix}
  (z_jI_d-J_d)^{-1}    &   0 \\
    0  &  0
\end{bmatrix}.
\end{equation}
It was shown in \cite{IDS15, Yin16} that $|D_{(i, i)}|>0$, for $i = 1, \ldots, s$.  According to (\ref{eq:8-4}), (\ref{eq:2-13}), (\ref{eq:11-10-4}) and (\ref{eq:11-18-1}), we have
\begin{eqnarray}
U^{(1)} & = & \frac{1}{2\pi \sqrt{-1}} \sum^{q}_{j=1}\omega_j(z_j  B- A)^{-1}B U^{(0)}  =  S DS^{-1}U^{(0)} \nonumber \\
  & = & \left[(S_{(:, 1:s)}+S_{(:, t+1:n)} E_{(1)})D_{(1:s, 1:s)}, V_{(1)}\right]Z,
\end{eqnarray}
where $ E_{(1)} = D_{(t+1: n, t+1:n)} E_{(0)}(D_{(1:s, 1:s)})^{-1}$ and
$V_{(1)} =S DS^{-1}V_{(0)}$. Denote the QR decomposition of $U^{(k)} $ by $U^{(k)} = U_{k}R_{k}$. By induction, we have the following relationship
\begin{eqnarray}
U^{(k)} & = & S DS^{-1}U_{k-1} = S DS^{-1}U^{(k-1)}(R_{k-1})^{-1} \nonumber \\
 & = & \left[(S_{(:, 1:s)}+S_{(:, t+1:n)} E_{(k)})(D_{(1:s, 1:s)})^k, V_{(k)}\right]Z(R_{k-1}\ldots R_1)^{-1}, \label{eq:11-11-2}
\end{eqnarray}
where $ E_{(k)} = (D_{(t+1: n, t+1:n)})^k E_{(0)}((D_{(1:s, 1:s)})^{-1})^k$ and
$V_{(k)} =(S DS^{-1})^kV_{(0)}$.

Since $Z(R_{k-1}\ldots R_1)^{-1}$ is nonsingular, from (\ref{eq:11-11-2}) we conclude  that the columns of $S_{(:, 1:s)}+S_{(:, t+1:n)} E_{(k)}$ are in the subspace ${\rm span}\{U^{(k)}\}$. In particular, vectors $v_{j}^{(k)} = S_{(:, 1+d_0+\ldots+d_{j-1})}+ S_{(:, t+1:n)} (E_{(k)})_{(:, 1+d_0+\ldots+d_{j-1})} \in {\rm span}\{U^{(k)}\}$ for $j = 1,\ldots, l$. By the special structure of $(D_{(1:s, 1:s)})^{-1})^k$, we have
\begin{eqnarray}
\|S_{(:, 1+d_0+\ldots+d_{j-1})}-v_j^{(k)}\|_2 & = & \|S_{(:, t+1:n)} (E_{(k)})_{(:, 1+d_0+\ldots+d_{j-1})}\|_2 \nonumber \\
& = & \big(\dfrac{1}{|\tilde{f}(\lambda_{j})|}\big)^k\|S_{(:, t+1:n)}(D_{(t+1: n, t+1:n)})^k (E_{(0)})_{(:, 1+d_0+\ldots+d_{j-1})}\|_2 \nonumber \\
 & \leq &  \tau_j  \big(\dfrac{ |\tilde{f}(\lambda_{l^{\prime}})|}{|\tilde{f}(\lambda_{j})|}\big)^k,\label{eq:11-11-3}
\end{eqnarray}
where $ \tau_j = \|S_{(:, t+1:n)}\|_2\| (E_{(0)})_{(:, 1+d_0+\ldots+d_{j-1})}\|_2$.

Moreover,
\begin{eqnarray}
\|(I_n - Q_{(k)})S_{(:, 1+d_0+\ldots+d_{j-1})}\|_2 & = &  \min_{v\in {\rm span}\{U^{(k)}\}} \|S_{(:, 1+d_0+\ldots+d_{j-1})}-v \|_2 \nonumber \\
 & \leq & \tau_j  \big(\dfrac{ |\tilde{f}(\lambda_{l^{\prime}})|}{|\tilde{f}(\lambda_{j})|}\big)^k.
\end{eqnarray}
\end{proof}

Note that  $l^{\prime} >l$, which means $\lambda_{l^{\prime}}$  is located outside $\Gamma$. Suppose the $ |\tilde{f}(\lambda_{l^{\prime}})|$ is about $1.0\times 10^{-3}$, we can expect that there exists a vector $v_j^{(k)}$ in  ${\rm span}\{U^{(k)}\}$ such that $\|S_{(:, 1+d_0+\ldots+d_{j-1})}-v_j^{(k)}\|_2\to 0$ at a rate of $10^{-3k}$. 

Let $P_{(k)}$ be  the oblique projector onto the subspace ${\rm span}\{U^{(k)}\}$ and orthogonal to the left subspace generated by a random matrix in the $k$th iteration in Algorithm \ref{alg:6}. Define approximate operators $A_k = P_{(k)}AQ_{(k)}$ and $B_k = P_{(k)}BQ_{(k)}$. The following theorem gives an upper bound for the residual norm of the exact eigenpair with respect to the approximate operator pair $(A_k, B_k)$.

\begin{theorem}\label{Th:11-11-1}
Let $\sigma^{(k)} = \|P_{(k)}(A-\lambda B)(I_n-Q_{(k)})\|_2$.  Then the residual norms of $(\lambda_j, S_{(:, 1+d_0+\ldots+d_{j-1})}), j = 1,\ldots, l$, for the approximate operator pair  $(A_k, B_k)$ saftisfy
\begin{equation}\label{eq:11-11-5}
\|(A_k-\lambda_j B_k)S_{(:, {(:, 1+d_0+\ldots+d_{j-1})})} \|_2\leq  \sigma^{(k)} \tau_j  \big(\dfrac{ |\tilde{f}(\lambda_{l^{\prime}})|}{|\tilde{f}(\lambda_{j})|}\big)^k
\end{equation}
\end{theorem}
\begin{proof}
Similar to the proof of Lemma 2 in \cite{IDS15} and Theorem 4.7 in \cite{saad}, we have
\begin{eqnarray}
\|(A_k-\lambda_j B_k)S_{(:, 1+d_0+\ldots+d_{j-1})} \|_2 & = & \|P_{(k)}(A-\lambda_j B)(I_n-Q_{(k)})S_{(:, 1+d_0+\ldots+d_{j-1})}\|_2 \nonumber\\
 & = & \|P_{(k)}(A-\lambda_j B)(I_n-Q_{(k)})(I_n-Q_{(k)})S_{(:, 1+d_0+\ldots+d_{j-1})}\|_2 \nonumber\\
 & \leq & \sigma^{(k)} \|(I_n-Q_{(k)})S_{(:, 1+d_0+\ldots+d_{j-1})}\|_2. \nonumber
\end{eqnarray}
By (\ref{eq:11-10-3}), we establish our result (\ref{eq:11-11-5}).
\end{proof}

Theorem \ref{Th:11-11-1}  says that the residual norm associated with the exact eigenpair  $(\lambda_j, S_{(:, 1+d_0+\ldots+d_{j-1})})$  converges  at a rate of $\big( |\tilde{f}(\lambda_{l^{\prime}})|/|\tilde{f}(\lambda_{j})|\big)$ with respect to the iteration counts.

\section{Numerical Experiments}\label{sec:experiments}
In this section, we present some numerical experiments to demonstrate the numerical performance of our new non-Hermitian FEAST algorithm (RFEAST). The experiments are organized into threes sets. The first set aims at demonstrating the convergence behavior of our new method. The second set is devoted to comparing our technique  with another non-Hermitian variant of FEAST, that is the BFEAST algorithm (Algorithm \ref{alg:5}). In the last set, we would like to compare our RFEAST method with the \textsc{Matlab} built-in function \texttt{eig}. 
 For the approximation eigenpairs $(\tilde{\lambda}_i,\tilde{{\bf x}}_i)$, define the relative residual norms
\begin{equation}\label{eq:11-22-1}
r_i = \frac{\| A\tilde{{\bf x}}_i - \tilde{\lambda}_i  B\tilde{{\bf x}}_i\|_2}{\| A \tilde{{\bf x}}_i\|_2+\| B \tilde{{\bf x}}_i\|_2}.
\end{equation}
We use the maximum relative residual norm defined as $\texttt{Res} = \max_{1\leq i \leq s}r_i$ to assess the accuracy achieved by the test methods.
All computations are carried out in \textsc{Matlab} version R2014b on a MacBook with an Intel Core i5 2.5 GHz processor and 8 GB RAM.

The test matrices presented in \textsc{Table} \ref{Tab:5-1} are available from the Matrix Market collection\footnote{http://math.nist.gov/MatrixMarket/}. They are the real-world problems from scientific and
engineering applications. All test problems are non-Hermitian. The first four test problems are generalized eigenvalue problems and the last two test problems are standard. The region of interest for each test problem is a circle with center at $\gamma$ and radius $\rho$. The value of $s$ is the number of eigenvalues inside the target region. In all experiments, we use the Gauss-Legendre quadrature with $16$ quadrature nodes to compute the approximate projection $U$ (see (\ref{eq:2-13})). The generalized shifted linear systems (see (\ref{eq:11-9-1})) involved are computed by direct method.
We first use the \textsc{Matlab} function \texttt{lu} to compute the LU
decomposition of $ A-z_j B, j =1 ,2,\ldots, q$, and then perform the triangular
substitutions to get the corresponding solutions.
\begin{table}
\centering
\caption{Test problems from Matrix Market that are used in our experiments.}
\footnotesize{
%{\newcommand{\q}[1]{\mc{1}{}{\small\tt #1}}
\noindent
\begin{tabular}{|c|c|c|c|c|c|}
\hline
No.&Problem & Type & $n$ & Region: $(\gamma, \rho)$ & $s$  \\ \hline
\hline
 1 & BFW782& gen. &782  &$(-6.0\times 10^{5}, 3.0\times 10^{5})$ &230 \\ \hline
 2& DWG961 & gen. & 961  &$(5.0\times 10^2, 2.0\times 10^2)$ &157 \\ \hline
 3 & UTM1700& gen. & 1700 & $(4.0, 1.0)$ & 96\\ \hline
 4 & MHD4800&gen. &4800  & $(-6.0, 3.0)$ & 169 \\ \hline
 5 &OLM5000 & stand. &5000 &$(-1.0\times 10^4, 0.6\times 10^4)$ & 204 \\ \hline
 6 & DW8192&stand. & 8192& $(1.0, 0.2)$ &270 \\ \hline
\end{tabular}}
\label{Tab:5-1}
\end{table}

\subsection{The convergence behaviour}
The FEAST algorithm is a stable and fast technique \cite{kramer, TP13}. It was formulated for the Hermitian problems \cite{polizzi}. The goal of our work is to adapt FEAST to the non-Hermitian cases. Meanwhile, we hope that our method retains the effectiveness of the FEAST algorithm. The objective of this experiment is two-fold. First, we would like to validate the convergence properties analysed in Section 4. Second, we would like to demonstrate the influence of the size of starting vectors on our new method.

In each iteration there are $t-s$ spurious eigenvalues. The spurious eigenvalues outside the target region can be easily detected according to their coordinates. For the spurious eigenvalues inside the target region, in \cite{YCY14} the authors introduced a tolerance $\eta$ to filter them. The idea behind is that the spurious eigenvalues can not achieve high accuracy; as the iteration process proceeds, there will be a gap in accuracy between the desired eigenpairs and the spurious ones. If the relative residual norm of an eigenpair is less that $\eta$, then the eigenpair is viewed as desired one and referred as filtered eigenpair. In the experiment, we set the filtering tolerance $\eta = 1.0\times 10^{-2}$.
In \textsc{Fig} \ref{Fig:5-1}, we plot the  \texttt{Res}'s from the iteration that the number of filtered eigenpairs attains $s$ for the first time to the $10$th iteration for the cases  $t = \lceil 1.2s \rceil$ and $t = \lceil 1.5s \rceil$, respectively. Here, we assume that the number $s$ of eigenvalues inside the region of interest is known.  Theorem  \ref{Th:11-11-1} tells us the residual norm will converge with the factor  $|\tilde{f}(\lambda_{l^{\prime}})|/|\tilde{f}(\lambda_{i})|$ for the exact eigenpair $(\lambda_i,{{\bf x}}_i)$ with respect to the iteration counts. \textsc{Fig} \ref{Fig:5-1} shows the maximum relative residual norm \texttt{Res} decreases monotonically, as expected, until the accuracy can not be further improved.   On the other hand, a larger subspace size $t$ leads to a smaller $|\tilde{f}(\lambda_{l^{\prime}})|$, and then leads to faster convergence. Taking the Problem 2 as an example, it is clear to see that our method converges almost linearly with a factor for both $t = \lceil 1.5s \rceil$ and $t = \lceil 1.2s \rceil$.   Precisely, the convergence rate is ablout $1.0\times 10^{-4}$  for the former case and is about  $1.0\times10^{-2}$ for the latter.  To converge to the minimum residual norm, which is about $1.0\times 10^{-13}$, 
it needs $4$ iterations when we take the size $t$ to  $\lceil 1.5s \rceil$, but 10 iterations  are required for the case $t = \lceil 1.2s \rceil$.
 
Increasing the value of $t$ will lead to a faster convergence rate, however, it also results in a considerable increase in computational cost in each iteration since $t$ represents the number of the right-hand sides in each shifted linear system involved (see (\ref{eq:11-9-1})).

\begin{figure}
\begin{center}
\includegraphics[width=13.0cm]{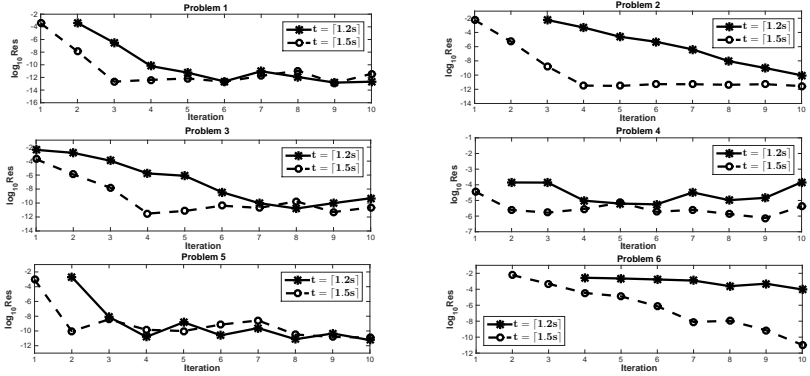}
\caption{The maximum relative residual norms in different iterations.}
\label{Fig:5-1}
\end{center}
\end{figure}

\subsection{Comparisons with BFEAST}
Both RFEAST  and  BFEAST \cite{YCY14} aim  to make the FEAST algorithm applicable for the non-Hermitian problems. The only difference between the two non-Hermitian FEAST methods is the choice of the left subspace. In BFEAST, the left subspace is spanned by $BU$ (see (\ref{eq:11-5-2})), while in our method the left subspace is spanned by a random matrix. The dominant work of both methods is  computing the approximate projection $U$ (see (\ref{eq:2-13})). Thus the computational cost required by both non-Hermitian FEAST algorithms in each iteration is almost the same. Due to this,  in this experiment we compare the numerical performance of the two methods through the accuracy achieved in each iteration. 

\begin{figure}
\begin{center}
\includegraphics[width=13.0cm]{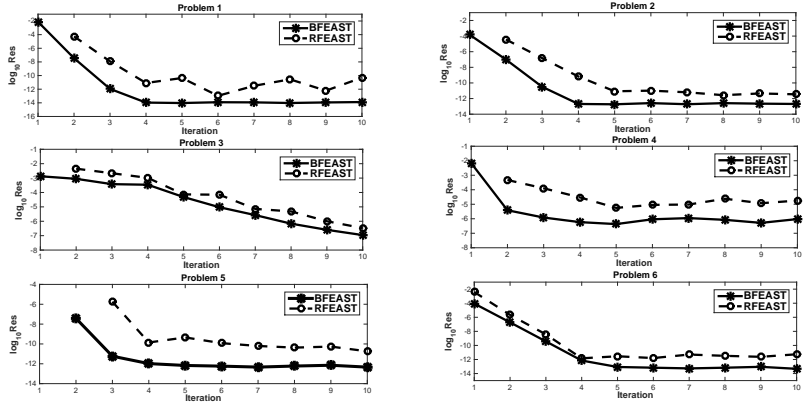}
\caption{The convergence behavior of two non-Hermitian FEAST algorithms.}
\label{Fig:5-2}
\end{center}
\end{figure}

In \cite{YCY14}, the authors presented a technique to select a suitable size of the starting vectors $Y$ for the BFEAST algorithm. To facilitate the comparisons, here we also use this technique to start our method. We depict  \texttt{Res}'s computed by the two test methods from the iteration that the number of filtered eigenpairs attains $s$ for the first
time to the 10th iteration in \textsc{Fig} \ref{Fig:5-2}. As with the previous experiment, the filtering tolerance $\eta$ is also taken to $1.0\times 10^{-2}$.
 The convergence curves of two non-Hermitian FEAST methods are almost parallel, which means the two methods converge with almost the same rate. We have shown in Theorem \ref{Th:11-11-1} that the upper bound for the residual norms of exact eigenpairs  $(\lambda_j, {\bf x}_j)$ are $ \sigma^{(k)} \tau_j  \big( |\tilde{f}(\lambda_{l^{\prime}})|/|\tilde{f}(\lambda_{j})|\big)^k$ in the $k$th iteration (see (\ref{eq:11-11-5})). Recall that our method shares the same right subspace with the BFEAST algorithm. In view of the proof of Theorem \ref{Th:11-11-1}, we are able to establish a similar upper bound  for the BFEAST algorithm simply via replacing the oblique projector $P_{(k)}$ in the expression of $\sigma^{(k)}$ with $Z_{(k)}$, where $Z_{(k)}$ is the oblique projector onto ${\rm span}\{U^{(k)}\}$ and orthogonal to the left subspace $B\mathcal{K}$. More precisely, we can write the upper bound for BFEAST as $ \kappa^{(k)} \tau_j  \big( |\tilde{f}(\lambda_{l^{\prime}})|/|\tilde{f}(\lambda_{j})|\big)^k$, where $ \kappa^{(k)} = \|Z_{(k)}(A-\lambda B)(I_n-Q_{(k)})\|_2$. Therefore, the two methods have almost the same convergence rate, which is $|\tilde{f}(\lambda_{l^{\prime}})|/|\tilde{f}(\lambda_{j})|$ for the eigenpair $(\lambda_j, {\bf x}_j)$.  This can interpret why  two non-Hermitian FEAST algorithms exhibit essentially the same convergence behavior.

On the other hand, it can be seen from \textsc{Fig} \ref{Fig:5-2}  that BFEAST performs better than our method in all test problems in terms of accuracy.  The only difference in the upper bounds between the two methods is the constants  $\kappa^{(k)}$ and $\sigma^{(k)}$, due to the different choices of the left subspaces. The BFEAST algorithm works better than our method possibly because the constant $\kappa^{(k)}$ in the BFEAST algorithm is smaller than $\sigma^{(k)}$  in our method, and therefore the upper bound in BFEAST is sharper than the one in our our method.

\subsection{Comparisons with \textsc{Matlab}'s \texttt{eig} function}
In this experiment, we compare our method with the \textsc{Matlab} built-in function \texttt{eig} in terms of timing. Since the target eigenvalues are the interior ones of non-Hermitian problems, when using \texttt{eig} to compute the eigenvalues inside the regions presented in \textsc{Tab} \ref{Tab:5-1}, we have to first compute all eigenvalues in dense format and then select the target eigenvalues according to their coordinates. 
In our method, we set  the convergence tolerance $\epsilon$ to $1.0\times 10^{-8}$ and 
take the parameter $\texttt{max\_iter}=10$.

%\begin{table}
%\centering
%\caption{Comparison of   \texttt{eig} and our method in terms of timing.}
%\footnotesize{
%%{\newcommand{\q}[1]{\mc{1}{|l|}{\small\tt #1}}
%\noindent
%\begin{tabular}{|c|cc|cc|}
%\hline
%\multirow{2}{*}{No.}   & \multicolumn{2}{c|}{\texttt{eig}}  &  \multicolumn{2}{c|}{our method}  \\ \cline{2-5}
% &   Time (sec.) & $\texttt{Res}$  &  Time (sec.) & $\texttt{Res}$ \\
%\hline
%\hline
%1& 9.80&$4.04\times 10^{-14}$ &5.18 &$1.17\times 10^{-13}$ \\
%2&14.13 &$2.02\times 10^{-12}$ &9.74 & $3.00\times 10^{-12}$ \\
%3&68.40 &$1.75\times10^{-13}$ &  52.21 & $1.50\times10^{-11}$\\
%4 &$2719.36$ & $1.44\times 10^{-7}$ & 112.16&$3.04\times10^{-6}$ \\
%5 &2397.71 & $3.04\times 10^{-12}$& 67.58 &$5.07\times 10^{-12}$ \\
%6 &77653.86 &$9.62\times10^{-13}$ & 486.22& $1.49\times 10^{-12}$ \\ \hline
%\end{tabular}}
%\label{Tab:5-2}
%\end{table}

\begin{table}
\centering
\caption{Comparison of   \texttt{eig} and our method in terms of timing.}
\footnotesize{
%{\newcommand{\q}[1]{\mc{1}{|l|}{\small\tt #1}}
\noindent
\begin{tabular}{|c|c|c|}
\hline
%\multirow{2}{*}{No.}   & \multicolumn{2}{c|}{\texttt{eig}}  &  \multicolumn{2}{c|}{our method}  \\ \cline{2-5}
No. &  \texttt{eig}   &  Our method  \\
\hline
\hline
1& 13.87 &3.22  \\
2&14.13  &10.94  \\
3&68.40  &  35.96 \\
4 &$2719.36$& 180.07 \\
5 &2397.71 & 27.58  \\
6 &77653.86  & 398.82 \\ \hline
\end{tabular}}
\label{Tab:5-2}
\end{table}

The amount of time, which is measured in seconds, required by \texttt{eig} and our RFEAST algorithm is reported in \textsc{Table} \ref{Tab:5-2}. It is clear to see that our method is much faster than the \textsc{Matlab} function \texttt{eig}, although the parallelism offered by our method is not used in the tests. The difference in CPU times is more obvious when the size of test problem grows larger. Therefore, our method is much more efficient than the \textsc{Matlab}'s \texttt{eig} function.

\section{Conclusions}  In this work, we have developed a new scheme to make the FEAST algorithm  applicable for the non-Hermitian problems. The key step  is that  the left subspace used to extract the desired eigenpairs in our method is spanned by a random matrix. Theoretical analysis shown that our method can deal with the non-Hermitian cases. The resulting method retains the feature of parallelism offered by the original FEAST algorithm and does not increase the computational cost. The convergence properties of our new method were also investigated. Numerical experiments were reported to demonstrate the numerical performance of our new method and to validated the convergence analysis.

\section{Acknowledgment} I would like to thank Professor Raymond H. Chan, my thesis advisor, at The Chinese University of Hong Kong and Professor Man-Chung Yeung at University of Wyoming for their help and fruitful discussions in the preparation of this paper.

\end{document}